\documentclass[reqno]{amsart}
\usepackage{amsmath,amsthm,amsfonts,amssymb,amscd}
\usepackage{graphicx,color}	
\usepackage{enumerate,tikz}
\usepackage{mathtools,mathdots}
\usepackage[ocgcolorlinks, linkcolor=red!90!black, citecolor=blue!80!black]{hyperref}

\DeclareFontFamily{U}{wncy}{}
\DeclareFontShape{U}{wncy}{m}{n}{<->wncyr10}{}
\DeclareSymbolFont{mcy}{U}{wncy}{m}{n}
\DeclareMathSymbol{\shift}{\mathbin}{mcy}{"58}

\newtheorem{thm}{Theorem}[section]
\newtheorem*{thm*}{Theorem}
\newtheorem{lem}[thm]{Lemma}
\theoremstyle{definition} \newtheorem{defn}[thm]{Definition}
\newtheorem{ex}[thm]{Example}
\newtheorem*{lem*}{Lemma}
\newtheorem{rem}[thm]{Remark}
\newtheorem*{conj*}{Conjecture}

\newtheorem{cor}[thm]{Corollary}
\newtheorem{conj}[thm]{Conjecture}
\newtheorem*{cor*}{Corollary}

\newtheorem*{warn*}{Warning}

\newcommand{\C}{\mathcal{C}}
\newcommand{\N}{\mathbb{N}}
\newcommand{\Z}{\mathbb{Z}}

\newcommand{\D}{\mathcal{D}}
\newcommand{\Q}{\mathbb{Q}}

\newcommand{\B}{\mathcal{B}}

\renewcommand{\dim}{\operatorname{dim}}

\renewcommand{\epsilon}{\varepsilon}

\newcommand{\Gr}{\operatorname{Gr}}

\newcommand{\st}{\operatorname{st}}
\newcommand{\rank}{\operatorname{rank}}

\renewcommand{\P}{\mathcal{P}}
\renewcommand{\Q}{\mathcal{Q}}

\renewcommand{\th}{\textsuperscript{th}\,}

\newcommand{\eqdef}{\overset{\text{def}}{=}}

\newcommand{\height}{\operatorname{ht}}
\newcommand{\touch}{\operatorname{tch}}

\begin{document}

\date{\today}
\author{Brendan Pawlowski}
\thanks{The author was partially supported by grant DMS-1101017 from the NSF}
\title{Catalan matroid decompositions of certain positroids}

\begin{abstract}
A \emph{positroid} is the matroid of a matrix whose maximal minors are all nonnegative. Given a permutation $w$ in $S_n$, the matroid of a generic $n \times n$ matrix whose non-zero entries in row $i$ lie in columns $w(i)$ through $n+i$ is an example of a positroid. We enumerate the bases of such a positroid as a sum of certain products of Catalan numbers, each term indexed by the $123$-avoiding permutations above $w$ in Bruhat order. We also give a similar sum formula for their Tutte polynomials. These are both avatars of a structural result writing such a positroid as a disjoint union of matroids, each isomorphic to a direct sum of Catalan matroids and a matroid with one basis.
\end{abstract}

\maketitle

\section{Introduction}

Given a permutation $w \in S_n$, consider a generic $n \times 2n$ matrix $M_w$ whose nonzero entries in row $i$ are in columns $[w(i), i+n]$. Here $[a,b]$ denotes $\{a, a+1, \ldots, b\}$ for integers $a$ and $b$; we also write $[n]$ for $[1,n]$. For example,
\begin{equation*}
M_{2143} = \begin{bmatrix}
0 & * & * & * & * & 0 & 0 & 0\\
* & * & * & * & * & * & 0 & 0\\
0 & 0 & 0 & * & * & * & * & 0\\
0 & 0 & * & * & * & * & * & *
\end{bmatrix}.
\end{equation*}
Let $\P_w$ be the set of bases of the matroid associated to $M_w$. That is, $\P_w$ is the set of $I \in {[2n] \choose n}$ such that the $n \times n$ minor of $M_w$ in rows $[n]$ and columns $I$ is nonzero.

The matroid $\P_w$ belongs to (at least) two interesting classes of matroids. First, it is a \emph{transversal matroid}; see \cite{bonin-transversal-matroids} for an introduction. Take a collection $A = \{A_1, \ldots, A_n\}$ of finite sets. A \emph{transversal} of $A$ is a set $\{x_1, \ldots, x_n\}$ such that $x_i \in A_i$ for each $i$ and all the $x_i$ are distinct. The set of all transversals of $A$ is the set of bases for a matroid. Indeed, if $M_A$ is a generic matrix with $n$ rows whose nonzero entries in row $i$ are in columns $j \in A_i$, then the matroid of $M_A$ is exactly the transversal matroid of $A$. Thus, $M_w$ is the transversal matroid of the set collection $\{[w(i), i+n] : 1 \leq i \leq n\}$.

Second, $\P_w$ is a \emph{positroid}: the matroid of a real matrix whose maximal minors are all nonnegative. Let $\Gr(k, N)$ be the Grassmann variety of $k$-planes in $\mathbb{C}^N$. Given a rank $k$ positroid $P$ on $[N]$, Knutson, Lam, and Speyer considered the closure of the locus of points in $\Gr(k, N)$ having matroid $P$ \cite{positroidjuggling}. Among other nice properties, these \emph{positroid varieties} turn out to be exactly the images of Richardson varieties in the complete flag variety under the projection to $\Gr(k,N)$.

Given any set of intervals $S = \{[a_1, b_1], \ldots, [a_k, b_k]\}$ in $[N]$, taking the rowspans of matrices of the form $M_S$ gives a subset of $\Gr(k,N)$ whose closure is an irreducible variety called a \emph{rank variety}. Billey and Coskun showed that rank varieties are exactly the images of Richardson varieties under the projection from the variety of \emph{partial} flags $F_1 \subseteq \cdots \subseteq F_k \subseteq \mathbb{C}^N$, where $\dim F_i = i$   \cite{rank-varieties}. Every rank variety is therefore a positroid variety, and in particular, $\P_w$ is a positroid.

Our main results concern the size and structure of $\P_w$. An \emph{anti-fixed point} of $w \in S_n$ is a number $i \in [n]$ such that $w(i) = n-i+1$. Define a permutation statistic
\begin{equation*}
g(w) = C_{\ell_1+1} \cdots C_{\ell_k+1},
\end{equation*}
where $\ell_1, \ldots, \ell_k$ are the lengths of the maximal runs of consecutive anti-fixed points in $w$, and $C_j$ is the $j$\th Catalan number. For example, in $w = 8697\underline{5}34\underline{21}$ we have underlined the maximal runs of anti-fixed points, and $g(w) = C_2 C_3 = 10$. Write $\leq$ for the strong Bruhat order on $S_n$.

\begin{thm} \label{thm:Pw-formula}
$\P_w$ has size \hspace{-0.5cm} $\displaystyle \sum_{\substack{v \geq w \\ \text{$v$ avoids $123$}}} \hspace{-0.3cm}  g(v)$ for any $w \in S_n$.
\end{thm}

Here, a permutation \emph{avoids} $123$ if it has no (not necessarily consecutive) increasing subsequence of length $3$. In the special case that $w = w_0 = n(n-1)\cdots 1$, Theorem~\ref{thm:Pw-formula} reads $\#\P_{w_0} = C_{n+1}$. In fact, $\P_{w_0}$ is isomorphic to the rank $n+1$ Catalan matroid $\C_{n+1}$ defined by Ardila, whose bases are the Dyck paths of length $2n+2$, each path viewed as the set of its upsteps \cite{catalan-matroid}.

Theorem~\ref{thm:Pw-formula} arises from a stronger structural result for $\P_w$ (cf. Theorem~\ref{thm:matroid-structure} below).

\begin{thm} \label{thm:Pw-structure} There is a partition of ${[2n] \choose n}$ into sets $\Q_v$ indexed by $123$-avoiding permutations $v$ such that for any $w \in S_n$,
\begin{itemize}
\item $\P_w$ is the disjoint union $\displaystyle \hspace{-0.4cm} \bigsqcup_{\substack{v \geq w \\ \text{$v$ avoids $123$}}} \hspace{-0.35cm} \Q_v$
\item If $v$ has runs of consecutive anti-fixed points of lengths $\ell_1, \ldots, \ell_k$, then $\Q_v$ is isomorphic to a direct sum of the Catalan matroids $\C_{\ell_1+1}, \ldots, \C_{\ell_r+1}$ plus a matroid with one basis. In particular, $\#\Q_v = g(v)$.
\end{itemize}
\end{thm}

In Section~\ref{sec:standardization}, we use a bijection of Krattenthaler between 123-avoiding permutations and Dyck paths to prove Theorem~\ref{thm:Pw-formula} in the case where $w$ is the identity permutation. This special case will be useful in proving Theorem~\ref{thm:Pw-structure}, which we do in Section~\ref{sec:Pw-structure}. In Section~\ref{sec:tutte-polynomial}, we give a formula for the Tutte polynomial of $\P_w$ along the lines of Theorem~\ref{thm:Pw-formula}. Section~\ref{sec:diagram-matroids} concludes with some conjectures about a related family of matroids also indexed by permutations.

\subsection*{Acknowledgements}

I would like to thank Sara Billey, Zach Hamaker, Vic Reiner, Jose Samper, Jair Taylor, and Alex Woo for helpful comments and discussions.

\section{Standardizing lattice paths to Dyck paths}
\label{sec:standardization}

Given a positive integer $n$, a \emph{Dyck path} of length $2n$ is a lattice path from $(0,0)$ to $(2n,0)$ which only uses steps $(1,1)$ (\emph{upsteps}) or $(1,-1)$ (\emph{downsteps}), and which never goes below the line $y = 0$. Let $\D_n$ be the set of Dyck paths of length $2n$. It is well-known that $\#\D_n$ is the $n$\th Catalan number $C_n$, and that this is also the number of $123$-avoiding $w \in S_n$.

If $w \in S_{2n}$ is the identity permutation, then $\P_w = {[2n] \choose n}$. In this case, Theorem~\ref{thm:Pw-formula} reads
\begin{equation} \label{eq:123-avoiding-identity}
\sum_{\substack{v \in S_n\\ \text{$v$ avoids $123$}}} C_{\ell_1+1}\cdots C_{\ell_k+1} = {2n \choose n},
\end{equation}
where $\ell_1, \ldots, \ell_k$ are the lengths of runs of anti-fixed points of each $v$.

Here is a similar identity for Dyck paths. We can view any $I \in {[2n] \choose n}$ as a lattice path from $(0,0)$ to $(2n,0)$ by taking one step for each $i = 1, 2, \ldots, 2n$, either $(1,1)$ or $(1,-1)$ depending on whether $i \in I$ or $i \notin I$. We say such a lattice path has a \emph{peak} at step $i$ if step $i$ is an upstep and step $i{+}1$ is a downstep. The \emph{height} of an upstep $i$ in a Dyck path is the $y$-coordinate of its endpoint; that is, the number of upsteps (weakly) before $i$ minus the number of downsteps before $i$. By the height of a peak $i$ we will mean the height of the corresponding upstep.

\begin{defn}
A \emph{saw} in a lattice path is a maximal consecutive sequence of height $1$ peaks.
\end{defn}
Here, two peaks are \emph{consecutive} if their upsteps occur in positions $i$ and $i+2$ for some $i$. The following identity will be the Dyck path analogue of \eqref{eq:123-avoiding-identity}.
\begin{lem} \label{lem:saw-identity} For any $n$,
\begin{equation*} 
\sum_{D \in \D_n} C_{\ell_1+1}\cdots C_{\ell_k+1} = {2n \choose n},
\end{equation*}
where $2\ell_1, \ldots, 2\ell_k$ are the lengths of the saws of each Dyck path $D$.
\end{lem}

This identity is not hard to prove. Suppose $I \in {[2n] \choose n}$ is a lattice path. The \emph{standardization} of $I$ is the Dyck path $\st(I)$ obtained by replacing each maximal segment of $I$ below the $x$-axis with a saw of the same length.

\begin{ex}
If $n = 9$ and $I = \{1,5,6,7,9,10,14,16,17\}$, so $I$ is the lattice path
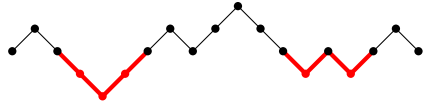
\begin{center}
\begin{tikzpicture}[scale=.3]
\tikzstyle{pathnode} = [minimum size=1.1mm, inner sep=0pt, fill, circle] 
\draw[red, ultra thick] (2, 0) node {} -- (3, -1) node[pathnode] {} -- (4, -2) node[pathnode] {} -- (5, -1) node[pathnode] {} -- (6, 0) node {};
\draw (0,0) node[pathnode] {} -- (1, 1) node[pathnode] {} -- (2, 0) node[pathnode] {};
\draw[red, ultra thick] (12, 0) node {} -- (13, -1) node[pathnode] {} -- (14, 0) node {};
\draw[red, ultra thick] (14, 0) node {}  -- (15, -1) node[pathnode] {} -- (16, 0) node {};
\draw (6, 0) node[pathnode] {} -- (7, 1) node[pathnode] {} -- (8, 0) node[pathnode] {} -- (9, 1) node[pathnode] {} -- (10, 2) node[pathnode] {} -- (11, 1) node[pathnode] {} -- (12, 0) node[pathnode] {};
\draw (14, 0) node[pathnode] {};
\draw (16, 0) node[pathnode] {} -- (17, 1) node[pathnode] {} -- (18, 0) node[pathnode] {}; 
\end{tikzpicture}
\end{center}
then $\st(I)$ is
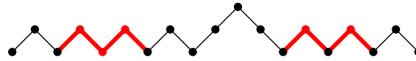
\begin{center}
\begin{tikzpicture}[scale=.3]
\tikzstyle{pathnode} = [minimum size=1.1mm, inner sep=0pt, fill, circle] 
\draw[red, ultra thick] (2, 0) node {} -- (3, 1) node[pathnode] {} -- (4, 0) node[pathnode] {} -- (5, 1) node[pathnode] {} -- (6, 0) node {};
\draw[red, ultra thick] (12, 0) node {} -- (13, 1) node[pathnode] {} -- (14, 0) node {};
\draw[red, ultra thick] (14, 0) node {} -- (15, 1) node[pathnode] {} -- (16, 0) node {};
\draw (0,0) node[pathnode] {} -- (1, 1) node[pathnode] {} -- (2, 0) node[pathnode] {};
\draw (6, 0) node[pathnode] {} -- (7, 1) node[pathnode] {} -- (8, 0) node[pathnode] {} -- (9, 1) node[pathnode] {} -- (10, 2) node[pathnode] {} -- (11, 1) node[pathnode] {} -- (12, 0) node[pathnode] {};
\draw (14, 0) node[pathnode] {};
\draw (16, 0) node[pathnode] {} -- (17, 1) node[pathnode] {} -- (18, 0) node[pathnode] {}; 
\end{tikzpicture}
\end{center}
where we have indicated maximal segments below the $x$-axis and their replacements in $\st(I)$ with bold red.
\end{ex}

\begin{proof}[Proof of Lemma~\ref{lem:saw-identity}]
Suppose $D \in \D_n$ is a Dyck path with saws of lengths $2\ell_1, \ldots, 2\ell_k$. The set $\st^{-1}(D)$ then has size $C_{\ell_1+1} \cdots C_{\ell_k+1}$. Indeed, the members of $\st^{-1}(D)$ are obtained from $D$ by replacing each saw of length $2\ell_i$ with an arbitrary lattice path of the same length which starts and ends on the $x$-axis and stays below $y = 1$. Prepending a downstep and appending an upstep shows that such lattice paths are in bijection with Dyck paths of length $2\ell_i + 2$. Thus Lemma~\ref{lem:saw-identity} reflects the partition of ${[2n] \choose n}$ into the fibers of the standardization map.
\end{proof}

Given Lemma~\ref{lem:saw-identity}, the identity \eqref{eq:123-avoiding-identity} would follow from a bijection from $123$-avoiding permutations to Dyck paths which turns anti-fixed points into peaks of height $1$. In fact, Krattenthaler has defined such a bijection \cite{krattenthaler-123-avoiding-dyck-path-bijection}. For the moment, view Dyck paths as proceeding from the southwest corner of the square $[n] \times [n]$ to the northeast, and remaining above the southwest-northeast diagonal. There is a partial order on Dyck paths where $D_1 \leq D_2$ if $D_1$ lies between $D_2$ and the diagonal of the square. If $w \in S_n$ is $123$-avoiding, define $K(w)$ to be the reverse of the minimal Dyck path which is northwest of the \emph{graph} of $w$, i.e. the set of points $\{(i, w(i)) : i \in [n]\} \subseteq [n] \times [n]$.

\begin{ex}
Say $w = 6475312$. The graph of $w$ is represented using $\times$'s, while $K(w)$ is the path in bold:

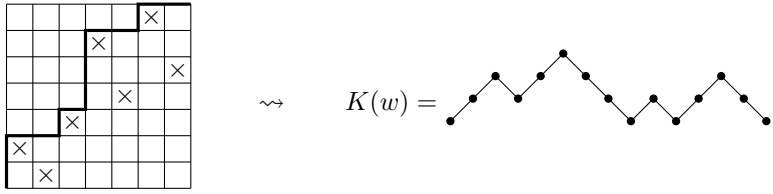
\begin{center}
\begin{tikzpicture}[scale=.35]
\draw (0,0) -- (7,0);
\draw (0,1) -- (7,1);
\draw (0,2) -- (7,2);
\draw (0,3) -- (7,3);
\draw (0,4) -- (7,4);
\draw (0,5) -- (7,5);
\draw (0,6) -- (7,6);
\draw (0,7) -- (7,7);
\draw (0,0) -- (0,7);
\draw (1,0) -- (1,7);
\draw (2,0) -- (2,7);
\draw (3,0) -- (3,7);
\draw (4,0) -- (4,7);
\draw (5,0) -- (5,7);
\draw (6,0) -- (6,7);
\draw (7,0) -- (7,7);

\node at (5.5, 6.5) {$\times$};
\node at (3.5, 5.5) {$\times$};
\node at (6.5, 4.5) {$\times$};
\node at (4.5, 3.5) {$\times$};
\node at (2.5, 2.5) {$\times$};
\node at (0.5, 1.5) {$\times$};
\node at (1.5, 0.5) {$\times$};

\draw[very thick] (0,0) -- (0,2) -- (2,2) -- (2,3) -- (3,3) -- (3,6) -- (5,6) -- (5,7) -- (7,7);
\end{tikzpicture} \qquad 
\raisebox{1.1cm}{$\leadsto \qquad K(w) = $ \raisebox{-.2cm}{\begin{tikzpicture}[scale=0.3]
\tikzstyle{pathnode} = [minimum size=1.1mm, inner sep=0pt, fill, circle] 
\draw (0,0) node[pathnode] {} -- (1, 1) node[pathnode] {} -- (2, 2) node[pathnode] {} -- (3, 1) node[pathnode] {} -- (4, 2) node[pathnode] {} -- (5, 3) node[pathnode] {} -- (6, 2) node[pathnode] {} -- (7, 1) node[pathnode] {} -- (8, 0) node[pathnode] {} -- (9, 1) node[pathnode] {} -- (10, 0) node[pathnode] {} -- (11, 1) node[pathnode] {} -- (12, 2) node[pathnode] {} -- (13, 1) node[pathnode] {} -- (14, 0) node[pathnode] {}; 
\end{tikzpicture}}}
\end{center}
\end{ex}

\begin{defn}
    A \emph{left-to-right minimum} of $w \in S_n$ is a position $i \in [n]$ such that $j < i$ implies $w(i) < w(j)$. A \emph{right-to-left maximum} is a position $i$ such that $j > i$ implies $w(i) > w(j)$.
\end{defn}

\begin{lem} \label{lem:LRmin-description}
Say $w \in S_n$ avoids $123$ and $j \in [n]$. Then $j$ is a left-to-right minimum if and only if $w(j) \leq n-j+1$, a right-to-left maximum if and only if $w(j) \geq n-j+1$, and an anti-fixed point if and only if it is both.
\end{lem}

\begin{proof}
Suppose $w(j) \leq n-j+1$ but $j$ is not a left-to-right minimum, so there is $i < j$ with $w(i) < w(j)$. Since $w$ avoids $123$, every $k$ such that $w(j) < w(k)$ must be in $[j] \setminus \{i\}$. But there are at least $j$ such values of $k$ given that $w(j) \leq n-j+1$, so this is impossible by the pigeonhole principle. Likewise, if $w(j) \geq n-j+1$, then $j$ is a right-to-left maximum. Every entry of $w$ is either a left-to-right minimum or a right-to-left maximum (a counterexample would yield a $123$ pattern), so the converses hold as well.
\end{proof}

The Dyck path $K(w)$ can now be described as follows. Say $1 = i_1 < \cdots < i_k$ are the left-to-right minima of $w$. Set $w(i_0) = n+1 = i_{k+1}$. Using $U$ for an upstep and $D$ for a downstep, 
\begin{equation} \label{eq:krattenthaler-bijection}
K(w) = U^{w(i_0) - w(i_1)} D^{i_2 - i_1} U^{w(i_1) - w(i_2)} D^{i_3 - i_2} \cdots U^{w(i_{k-1}) - w(i_k)} D^{i_{k+1} - i_k}.
\end{equation}

\begin{lem} \label{lem:peak-height}
    Suppose $w \in S_n$ avoids $123$. Then $j$ is a left-to-right minimum of $w$ if and only if $K(w)$ has a peak at $n-w(j)+j$, in which case the peak has height $n+2-w(j)-j$.
    \end{lem}

\begin{proof}
    It is clear from \eqref{eq:krattenthaler-bijection} that the left-to-right minima of $w$ correspond to the peaks of $K(w)$. The peak corresponding to $i_p$ is preceded by $\sum_{q=1}^p (w(i_{q-1}) - w(i_q)) = n+1 - w(i_p)$ upsteps and by $\sum_{q=2}^p (i_{q} - i_{q-1}) = i_p - 1$ downsteps. The position and height of this peak are, respectively, the sum and difference of these two counts: $n-w(i_p)+i_p$ and $n+2-w(i_p)-i_p$.
\end{proof}

The next two corollaries follow using Lemmas \ref{lem:LRmin-description} and \ref{lem:saw-identity}.
\begin{cor} If $w$ avoids $123$ and has runs of anti-fixed points of lengths $\ell_1, \ldots, \ell_k$, then $K(w)$ has saws of lengths $2\ell_1, \ldots, 2\ell_k$.
\end{cor}

\begin{cor} \label{cor:123-avoiding-identity} For any $n$,
\begin{equation*}
\sum_{\substack{v \in S_n\\ \text{$v$ avoids $123$}}} C_{\ell_1+1}\cdots C_{\ell_k+1} = {2n \choose n},
\end{equation*}
where $\ell_1, \ldots, \ell_k$ are the lengths of runs of anti-fixed points of each $v$.
\end{cor}

\section{The structure of $\P_w$}
\label{sec:Pw-structure}

\begin{defn}
    The $n$\th \emph{Catalan matroid} has groundset $[n]$ and bases
    \begin{equation*}
        \C_n \eqdef \{ \{i \in [n] : \text{$i$ an upstep of $D$}\}: D \in \D_n\}.
    \end{equation*}
\end{defn}

Ardila showed that $\C_n$ is indeed the set of bases of a matroid, and that this matroid can also be represented by a generic $n \times 2n$ matrix of the form
\begin{equation*}
A_{n} \eqdef \begin{bmatrix}
*       & 0      &      0 &      0 &   0    & \cdots & 0      & 0      & 0\\
*       & *      &      * &      0 &   0    & \cdots & 0      & 0      & 0\\
*       & *      &      * &      * &   *    & \cdots & 0      & 0      & 0\\
\vdots  & \vdots & \vdots & \vdots & \vdots & \ddots & \vdots & \vdots & \vdots\\
*       & *      &      * &      * &   *    & \cdots & *      & *      & 0
\end{bmatrix}
\end{equation*}
That is, for an $n$-subset $I$ of $[2n]$, the minor of this matrix in rows $[n]$ and columns $I$ is nonzero if and only if $I \in \C_n$.

Recall that $\P_w$ is the matroid of the matrix $M_w$ as defined in the introduction. Write $\P_n$ for $\P_{w_0}$ where $w_0 = n(n{-}1)\cdots 321 \in S_n$. Then $\P_n$ is represented by the $n \times 2n$ matrix
\begin{equation*}
M_{w_0} = \begin{bmatrix}
0      & 0      &  \cdots  & 0 & * & * & 0 & \cdots & 0      & 0\\
\vdots & \vdots &  \iddots  & * & * & * & * & \ddots & \vdots & 0\\
0      & 0      &  \iddots & * & * & * & * & \ddots & 0      & 0\\
0      & *      &  \cdots  & * & * & * & * & \cdots & *      & 0\\
*      & *      &  \cdots  & * & * & * & * & \cdots & *      & *
\end{bmatrix}
\end{equation*}
Deleting row $1$ and columns $1$ and $2n+2$ of $A_{n+1}$, then permuting columns appropriately, gives the matrix $M_{w_0}$. Hence $\P_n$ is isomorphic to $\C_{n+1}$. Specifically, say $\alpha : [2,2n+1] \to [2n]$ is the function sending $2, 3, \ldots, n+1$ to $n+1, n, n+2, n-1, \ldots, 2n, 1$. Then $D \in \C_{n+1}$ if and only if $\alpha(D \setminus \{1\}) \in \P_n$.
\begin{lem} \label{lem:reverse-perm-positroid} If $w_0 \in S_n$ is the reverse permutation, $\P_n$ is the set of $I \in {[2n] \choose n}$ such that $\#(I \cap [n-j+1,n-j]) \geq j$ for $1 \leq j \leq n$.
\end{lem}

\begin{proof}
    $D \in {[2n+2] \choose n+1}$ is the set of upsteps of a Dyck path of length $2n+2$ if and only if $[k]$ contains at least as many members of $D$ as of $[2n+2] \setminus D$, for each $k$. In fact, this only needs to hold for each odd $k$. Equivalently, $D \in \C_{n+1}$ if and only if $\#(D \cap [2j+1]) \geq j+1$ for $0 \leq j \leq n$. Setting $I = \alpha(D \setminus \{1\})$, this condition is equivalent to the lemma.
\end{proof}

We will need to consider versions of $\P_n$ on groundsets other than $[n]$, for which the following notation will be useful. Given a subset $X = \{x_1 < \cdots < x_k\}$ of $[n]$, write $Z_j X$ for the set
\begin{equation*}
\{n-x_j+1, \ldots, n-x_1+1, n+x_1, \ldots, n+x_j\}
\end{equation*}
Note that $Z_j X$ also depends on $n$, but we suppress that in the notation. We will abbreviate $Z_{\#X}(X)$ as $Z(X)$. Now for an interval $K \subseteq [n]$ of size $k$, let $f_{K,n}$ be the unique increasing function $[2k] \to Z(K)$. Finally, define $\P_{K,n}$ to be $f_{K,n}(\P_k)$. For example, $\P_{[3,4],7}$ is the matroid of a generic matrix
\begin{equation*} \left[
\begin{array}{cccccccccccccc}
0 & 0 & 0 & 0 & * & 0 & 0 & 0 & 0 & * & 0 & 0 & 0 & 0\\
0 & 0 & 0 & * & * & 0 & 0 & 0 & 0 & * & * & 0 & 0 & 0 
\end{array} \right].
\end{equation*}
Alternatively, we can give a description in the style of Lemma~\ref{lem:reverse-perm-positroid}: $\P_{K,n}$ consists of the $k$-subsets $I$ of $Z(K)$ such that $\#(I \cap Z_j K) \geq j$ for each $j$ in $[k]$.

Let $L(w)$ be the set of left-to-right minima of $w$ which are not right-to-left maxima, and $R(w)$ the set of right-to-left maxima which are not left-to-right minima. We can now state our main structural result for $\P_w$.

\begin{thm} \label{thm:matroid-structure} Say $v, w \in S_n$.
\begin{enumerate}[(a)]
\item If $v \leq w$ in Bruhat order, then $\P_w \subseteq \P_v$.
\item The sets $\Q_w \eqdef \P_w \setminus \bigcup_{v > w} \P_v$ are pairwise disjoint.
\item If $w$ contains $123$, then $\Q_w$ is empty.
\item If $w$ avoids $123$, let $A_1, \ldots, A_k$ be the maximal intervals in the set of anti-fixed points of $w$. Then
\begin{equation*}
\Q_w = \bigoplus_{i=1}^k \P_{A_i, n} \oplus \{w(L(w))\} \oplus \{n + R(w)\}.
\end{equation*}
\end{enumerate}
\end{thm}
Here, for two families of sets $\mathcal{F}$ and $\mathcal{G}$, we write $\mathcal{F} \oplus \mathcal{G}$ for the family $\{I \sqcup J : I \in \mathcal{F}, J \in \mathcal{G}\}$. That is, if $\mathcal{F}$ and $\mathcal{G}$ are sets of bases for two matroids, then $\mathcal{F} \oplus \mathcal{G}$ is the set of bases for the direct sum of the two matroids. Also, for a set $A$ and integer $n$, we let $n + A \eqdef \{i + n : i \in A\}$.

\begin{ex}
Take $w = 645312$, which avoids $123$. The runs of anti-fixed points occur in positions $1$ and $4$, and $L(w) = \{2,5\}$ and $R(w) = \{6,3\}$. Hence
\begin{equation*}
\Q_w = \P_{[1,1], 6} \oplus \P_{[4,4], 6} \oplus \{\{4,1\}\} \oplus \{\{12,9\}\}.
\end{equation*}
We have $\P_{[1,1],6} = \{\{6\}, \{7\}\}$ and $\P_{[4,4],6} = \{\{3\}, \{10\}\}$. So, $\Q_w$ consists of the four sets
\begin{equation*}
13469(12),\quad 13479(12),\quad 1469(10)(12),\quad 1479(10)(12).
\end{equation*}
\end{ex}

\begin{rem} \label{rem:Qw-description}
    The description of $\Q_w$ given by Theorem~\ref{thm:matroid-structure}(d) can be rephrased in the manner of Lemma~\ref{lem:reverse-perm-positroid}.  Let $A$ be the set of anti-fixed points of $w$, and define
    \begin{equation*}
    G(w) \eqdef w(L(w) \cup A) \cup (n + (R(w) \cup A)) = w(L(w)) \cup (R(w) + n) \cup Z(A).
    \end{equation*}
    Then $\Q_w$ consists of the $n$-subsets $I$ of $G(w)$ such that $\#(I \cap Z_j K) \geq j$ for each maximal interval $K \subseteq A$ and each $j \in [\#K]$. In particular, $\#(I \cap Z(A \cap [j])) \geq \#(A \cap [j])$ for any $j \in [n]$.

    Alternatively, $\Q_w$ is the set of bases of a matroid. For $w$ avoiding $123$, let $N_w$ be a generic matrix whose entries are zero except that
     \begin{itemize}
     \item The entries $(i, w(i))$ for $i \in L(w)$ are nonzero.
     \item The entries $(i, i+n)$ for $i \in R(w)$ are nonzero.
     \item Suppose $w$ has runs of anti-fixed points in positions $A_1, \ldots, A_k$. For each $p$, the submatrix of $N$ in rows $A_p$ and columns $Z(A_p)$ is $M_{w_0}$, where $w_0 \in S_{\#A_p}$.
     \end{itemize}
     Then $\Q_w$ is the matroid of $N_w$. For instance, if $w = 645312$ as above then
     \begin{equation*}
        \left[\begin{array}{cccccccccccc}
            0 & 0 & 0 & 0 & 0 & * & * & 0 & 0 & 0 & 0 & 0\\
            0 & 0 & 0 & * & 0 & 0 & 0 & 0 & 0 & 0 & 0 & 0\\
            0 & 0 & 0 & 0 & 0 & 0 & 0 & 0 & * & 0 & 0 & 0\\
            0 & 0 & * & 0 & 0 & 0 & 0 & 0 & 0 & * & 0 & 0\\
            * & 0 & 0 & 0 & 0 & 0 & 0 & 0 & 0 & 0 & 0 & 0\\
            0 & 0 & 0 & 0 & 0 & 0 & 0 & 0 & 0 & 0 & 0 & *
        \end{array} \right]
    \end{equation*}
    \end{rem}

To prove Theorem~\ref{thm:matroid-structure}, we begin with a characterization of positroids from \cite{oh-positroid-characterization}. An \emph{affine permutation} is a bijection $f : \Z \to \Z$ such that $f(i+n) = f(i)+n$ for some fixed $n$ (the \emph{quasiperiod} of $f$) and all $i \in \Z$. Notice that an affine permutation is determined completely by the word $f(1)f(2)\cdots f(n)$, and we will specify an affine permutation by this word. For example, $4721$ sends $4k+1 \mapsto 4k+4$ for any $k$, sends $4k+2 \mapsto 4k+7$, and so on.

An affine permutation $f$ is \emph{bounded} if $i \leq f(i) \leq n+i$ for each $i \in \Z$. Suppose $f$ is bounded and that exactly $k$ of the values $f(1), \ldots, f(n)$ exceed $n$. The \emph{juggling sequence} of $f$ is the sequence $(J_1, \ldots, J_n)$ of $k$-subsets of $[n]$ given by $J_i = \{f(j) - i + 1 : j < i\} \cap \N$. Finally, let $\chi$ be the cyclic shift permutation $23\cdots n1 \in S_n$.

\begin{defn} \label{def:positroid}
The \emph{positroid} associated to an $f$ as described above is the matroid on $[n]$ with bases
\begin{equation} \label{eq:positroid-characterization}
\left\{I \in {[n] \choose k} : \chi^{-i+1} I \geq J_i \text{ for all $i = 1, \ldots, n$}\right\},
\end{equation}
where $\{a_1 < \cdots < a_k\} \leq \{b_1 < \cdots < b_k\}$ if $a_i \leq b_i$ for all $i$.
\end{defn}
Postnikov \cite{postnikov-positroids} gave various combinatorial descriptions of positroids, and conjectured that Definition~\ref{def:positroid} agrees with the definition of positroid given in the introduction---this conjecture was proven in \cite{oh-positroid-characterization}. The description in terms of bounded affine permutations is due to Knutson, Lam, and Speyer \cite{positroidjuggling}.

For $w \in S_n$, let $f_w$ be the bounded affine permutation of quasiperiod $2n$ with
\begin{equation*}
f_w(i) = \begin{cases}
i+n & \text{if $1 \leq i \leq n$}\\
w(i)+2n & \text{if $n+1 \leq i \leq 2n$}
\end{cases}.
\end{equation*}
For instance, $f_{2143} = 5678(10)9(12)(11)$. 
\begin{thm}\label{thm:positroid-equals-rank} $\P_w$ is the positroid associated to $f_w$. \end{thm}
    \begin{proof}
        Let $\Pi_w^\circ$ be the set of $n$-planes in $\Gr(n, 2n)$ whose matroid is the positroid associated to $f_w$, and let $\Sigma_w^\circ$ be the set of $n$-planes in $\Gr(n, 2n)$ which are rowspans of matrices of the form $M_w$ whose nonzero entries are algebraically independent. It is shown in \cite[\S 4]{pawlowski-liu-conjecture-rank-varieties} that the Zariski closures $\overline{\Pi_w^\circ}$ and $\overline{\Sigma_w^{\circ}}$ are equal. Because $\Pi_w^\circ$ is defined by requiring certain Pl\"ucker coordinates on $\Gr(n, 2n)$ to be nonzero and the rest to be zero, it is locally closed, so $\overline{\Pi_w^{\circ}} \setminus \Pi_w^{\circ}$ is closed. This means that $\Sigma_w^{\circ}$ cannot be contained in $\overline{\Pi_w^{\circ}} \setminus \Pi_w^{\circ}$, because then its closure would be, contradicting $\overline{\Pi_w^\circ} = \overline{\Sigma_w^{\circ}}$. It follows that $\Sigma_w^{\circ} \cap \Pi_w^\circ$ is nonempty. Every member of $\Sigma_w^\circ$ has matroid $\P_w$, so this proves the theorem.
    \end{proof}

The juggling sequence $(J_1, \ldots, J_{2n})$ of $f_w$ is easy to describe: $J_1 = \cdots = J_{n+1} = [n]$, while $J_{n+j+1} = [n-j] \cup \{w([j]) + n-j\}$ for $j \in [n-1]$. This leads to a correspondingly simpler version of the test for membership in $\P_w$ given by Definition~\ref{def:positroid}. Given $I \in {[2n] \choose n}$ and some $j$, write $\chi^{-n-j}I = \{b_1 < \cdots < b_n\}$, and define $B_j(I) = \{b_{n-j+1}, \ldots, b_n\} - n + j$.

\begin{lem} \label{lem:positroid-simplification-1} A set $I \in {[2n] \choose n}$ is in $\P_w$ if and only if $B_j(I) \geq w([j])$ for $j = 1, \ldots, n$. \end{lem}

\begin{proof}
By Theorem~\ref{thm:positroid-equals-rank}, $I \in \P_w$ if and only if $\chi^{-i+1} I \geq J_i$ for $i \in [2n]$. This test is vacuous for $i \leq n+1$ since $J_i = [n]$. If $i = n+j+1$, it reads $\chi^{-n-j}I \geq [n-j] \cup \{w([j]) + n-j\}$, which is equivalent to $B_j(I) \geq w([j])$.
\end{proof}

Lemma~\ref{lem:positroid-simplification-1} can be simplified and rephrased in terms of Bruhat order on words.
\begin{defn} An \emph{injective word} on $\N$ is a word whose letters are all distinct, i.e. an injective function $v : [\ell] \to \N$ for some $\ell$. The \emph{Bruhat order} on injective words of length $\ell$ has $v \geq w$ if and only if $v([j]) \geq w([j])$ for $j \in [\ell]$.
\end{defn}
When restricted to permutations of $[\ell]$, the definition of Bruhat order above is sometimes called the \emph{tableau criterion}, and it agrees with the usual strong Bruhat order on permutations \cite[Theorem 2.6.3]{bjorner-brenti}.

Given
\begin{equation*}
I = \{i_1 < \cdots < i_p \leq n < i_{p+1} < \cdots < i_n\} \in {[2n] \choose n},
\end{equation*}
let $v_I$ be the injective word with $i_{p+1}, \ldots, i_n$ in positions $i_{p+1}-n, \ldots, i_n - n$ (in increasing order), and $i_p, \ldots, i_1$ in the remaining positions (in decreasing order). For example, if $n = 6$ and $I = \{1, 2, 5, 7, 10, 11\}$, then $v_I = 752(10)(11)1$.

\begin{lem} \label{lem:positroid-simplification-2} For any $w \in S_n$ and $I \in {[2n] \choose n}$, $I \in \P_w$ if and only if $v_I \geq w$. \end{lem}
\begin{proof} This will follow from Lemma~\ref{lem:positroid-simplification-1} if we show that $B_{j-1}(I) \subseteq B_j(I)$ for each $j$ and that $B_j(I) \setminus B_{j-1}(I) = \{v_I(j)\}$. For each $j \geq 0$, we have
\begin{equation*}
\chi^{-n-j} I = \{i_{p+q(j)} - n < \cdots < i_n - n < i_1 + n < \cdots < i_{p+q(j)-1} + n\} - j,
\end{equation*}
where $q(j) \geq 0$ is such that $i_1 < \cdots < i_{p+q(j)-1} \leq n+j < i_{p+q(j)} < \cdots < i_n$. We must have $n-(p+q(j))+1 \leq 2n - (n+j)$, or equivalently $p+q(j)-1 \geq j$. Therefore
\begin{equation} \label{eq:Bj-description}
B_j(I) = \{i_{p+q(j)-j} < \cdots < i_{p+q(j)-1}\}.
\end{equation}

There are two cases now. If $j = i_{p+r} - n$ for some $r \geq 1$, then $q(j) = q(j-1)+1 = r+1$. One can then see from \eqref{eq:Bj-description} that $B_j(I) \subseteq B_{j-1}(I)$ and that
\begin{equation*}
B_j(I) \setminus B_{j-1}(I) = \{i_{p+q(j-1)}\} = \{i_{p+r}\} = \{v_I(j)\}.
\end{equation*}

On the other hand, if $j \notin \{i_{p+1}-n, \ldots, i_n-n\}$ and $j \geq 1$, then $q(j) = q(j-1)$. Again \eqref{eq:Bj-description} shows that $B_j(I) \subseteq B_{j-1}(I)$, and now
\begin{equation*}
B_j(I) \setminus B_{j-1}(I) = \{i_{p+q(j-1)-j}\} = \{i_{p+q(j)-j}\}.
\end{equation*}

Since the sets $B_j(I)$ are nested and get larger by one element with each step, the word formed by the singletons $B_j(I) \setminus B_{j-1}(I)$ must be injective, and its entries are the members of $I$ in some order by \eqref{eq:Bj-description}. We have seen that the entries in positions $\{i_{p+1}-n, \ldots, i_n-n\}$ agree with those for $v_I$. Therefore to show that the remaining entries are $i_p, \ldots, i_1$, it suffices to show that they come in decreasing order. This follows from the fact that the function $j \mapsto p+q(j)-j$ is weakly decreasing, since $q(j+1)-q(j) \in \{0,1\}$ for each $j$.
\end{proof}

Lemma~\ref{lem:positroid-simplification-2} says that $\P_w$ is the inverse image in ${[2n] \choose n}$ of the order filter above $w$ in the poset of length $n$ injective words under the map $I \mapsto v_I$. The following dual perspective will also be useful. Given a fixed $I \in {[2n] \choose n}$, let $W_I \eqdef \{w \in S_n : I \in \P_w\}$. By Lemma~\ref{lem:positroid-simplification-2}, $W_I = \{w \in S_n : v_I \geq w\}$. Recall that part (b) of Theorem~\ref{thm:matroid-structure} claims that the sets $\P_w \setminus \bigcup_{v > w} \P_v$ are pairwise disjoint for $w \in S_n$, which is equivalent to the statement that $W_I$ has a unique maximal element.  If $v_I$ and $w_0 = n(n{-}1)\cdots 321$ have a greatest lower bound, it will be the unique maximal element of $W_I$. The poset of injective words with Bruhat order is not a lattice, but in fact the greatest lower bound exists in this case.

\begin{lem} \label{lem:word-meet} Let $v$ be an injective word of length $n$ and $w_0 = n(n{-}1)\cdots 321$. For each $j$, define
\begin{equation*}
s(j) = \#([v(j)] \cap v([j])) = \#\{1 \leq i \leq j : v(i) \leq v(j)\}.
\end{equation*}
Let $u \in S_n$ be such that $u(j) = v(j)$ if $v(j) \leq n-j + s(j)$, and whose other entries are the other members of $[n]$, in decreasing order. Then $u$ is a greatest lower bound for $v$ and $w_0$ in Bruhat order. \end{lem}

\begin{proof}
Define
\begin{equation*}
E_j = \min(v([j]), w_0([j])) = \{\min(b_1, n-j+1) < \cdots < \min(b_j, n)\},
\end{equation*}
where $v([j]) = \{b_1 < \cdots < b_j\}$. If the sets $E_j$ are nested, the corresponding injective word will be a greatest lower bound for $v$ and $w_0$, so we must show that $E_{j-1} \subseteq E_j$ and $E_j \setminus E_{j-1} = \{u(j)\}$ for each $j$. The proof will be similar to that of Lemma~\ref{lem:positroid-simplification-2}.

For each $j$, take  $r(j)$ maximal such that $b_{r(j)} \leq n-j+r(j)$, or $0$ if there is no such $r$. For a fixed $j$, write $v([j-1]) = \{b_1 < \cdots < b_{j-1}\}$. Then
\begin{equation*}
E_{j-1} = \{b_1 < \cdots < b_{r(j-1)} < n-j+r(j-1)+2 < \cdots < n\}.
\end{equation*}
Now we consider two cases.
\begin{itemize}
    \item Suppose $v(j) \leq n-j+s(j)$. Then $r(j) \geq s(j)$, so $r(j) = r(j-1)+1$ and
\begin{align*}
E_j &= \{b_1 < \cdots < b_{s-1} < v(j) < b_s < \cdots < b_{r(j-1)} < n-j + r(j)+1 < \cdots < n\}\\
    &= \{b_1 < \cdots < b_{s-1} < v(j) < b_s < \cdots < b_{r(j-1)} < n-j + r(j-1)+2 < \cdots < n\}\\
    &= E_{j-1} \cup \{v(j)\}.
\end{align*}

\item Suppose $v(j) > n-j+s(j)$. Then $r(j) \leq s(j)-1$. In this case we have $r(j) \leq r(j-1)$, and if $r(j) < i \leq r(j-1)$, then $b_i = n-j+i+1$. Therefore
\begin{align*}
&E_j = \{b_1 < \cdots < b_{r(j)} < n-j + r(j)+1 < \cdots < n\}\\
&E_{j-1} = \{b_1 < \cdots < b_{r(j)} < n-j+r(j)+2 < \cdots < n\},
\end{align*}
so $E_j = E_{j-1} \cup \{n-j+r(j)+1\}$.
\end{itemize}

In both cases we see that $E_{j-1} \subseteq E_j$, so the permutation $u$ defined by $E_j \setminus E_{j-1} = \{u(j)\}$ is a greatest upper bound for $w_0$ and $v$. Moreover, if $v(j) \leq n-j+s(j)$, then $u(j) = v(j)$. If on the other hand $j$ is such that $v(j) > n-j+s(j)$, then $u(j) = n-j+r(j)+1$; since $j \mapsto r(j)-j$ is a weakly decreasing function, we see that $u$ is weakly decreasing on such positions $j$, as claimed.
\end{proof}

\begin{cor} \label{cor:W_I-max} For any $I \in {[2n] \choose n}$, the set $W_I$ has a unique maximal element $u_I$. \end{cor}

\begin{proof} $u_I$ is the greatest lower bound of $v_I$ and $w_0$ given by Lemma~\ref{lem:word-meet}.
\end{proof}

\begin{rem} \label{rem:u_I-calc}
Calculating $u_I$ is simpler than Lemma~\ref{lem:word-meet} might lead one to believe, because the entries of $v_I$ in $[n]$ form a decreasing sequence, so in the case that $v_I(j)\leq n-j+s(j)$, we actually have $s(j) = 1$. Hence $u_I$ is the permutation in $S_n$ such that $u_I(j) = v_I(j)$ when $v_I(j) \leq n-j+1$, and whose other entries form a decreasing subsequence. For example, say $n = 9$ and $I = \{1, 3, 4, 6, 9, 10, 15, 16, 17\}$. Then $v_I = (10)9643(15)(16)(17)1$ and $u_I = 986437521$. The next lemma shows that $u_I$ is determined by even less information.
\end{rem}

\begin{lem} \label{lem:u_I-vs-v_I}
For any $I \in {[2n] \choose n}$, we have $u_I(j) < n-j+1$ if and only if $v_I(j) < n-j+1$ (and in this case $u_I(j) = v_I(j)$). Moreover, the permutation $u_I$ is uniquely determined by the set $\{(j, v_I(j)) : v_I(j) < n-j+1\}$.
\end{lem}

\begin{proof}
The description of $u_I$ in Remark~\ref{rem:u_I-calc} shows that $u_I$ is the union of two decreasing subsequences, and so it avoids the pattern $123$. A $123$-avoiding permutation $z \in S_n$ is uniquely determined by the pairs $(j, z(j))$ for which $z(j) < n-j+1$, because the other entries will be right-to-left maxima and come in decreasing order. Thus it suffices to prove the first claim.

By Remark~\ref{rem:u_I-calc}, if $v_I(j) < n-j+1$ then $u_I(j) = v_I(j)$. Suppose that $u_I(j) < n-j+1$ but that $u_I(j) \neq v_I(j)$. This implies $v_I(j) > n-j+1$. By the pigeonhole principle, there must be $k > j$ such that $u_I(k) > n-k+1$, and the minimal such $k$ must satisfy $u_I(k) > u_I(j)$. For such a $k$ we have $v_I(k) > n-k+1$. But then the construction of $u_I$ implies that $u_I(j)$ and $u_I(k)$ are part of the same decreasing subsequence. This is a contradiction, since $j < k$ and $u_I(j) < u_I(k)$.
\end{proof}

We now restate and prove Theorem~\ref{thm:matroid-structure}.

\begin{thm*}[Theorem \ref{thm:matroid-structure}] Say $v, w \in S_n$. Then
\begin{enumerate}[(a)]
\item If $v \leq w$ in Bruhat order, then $\P_w \subseteq \P_v$.
\item The sets $\Q_w = \P_w \setminus \bigcup_{v > w} \P_v$ are pairwise disjoint.
\item If $w$ contains $123$, then $\Q_w$ is empty.
\item If $w$ avoids $123$, say $w$ has runs of anti-fixed points $A_1, \ldots, A_k$. Then
\begin{equation} \label{eq:Qw-matroid}
\Q_w = \bigoplus_{i=1}^k \P_{A_i, n} \oplus \{w(L(w))\} \oplus \{n + R(w)\}.
\end{equation}
\end{enumerate}
\end{thm*}

\begin{proof} \hfill
\begin{enumerate}[(a)]
\item Immediate from Lemma~\ref{lem:positroid-simplification-2}.

\item We have $I \in \Q_w$ if and only if $w$ is a maximal element of $W_I = \{w \in S_n : I \in \P_w\}$, so this follows from Corollary~\ref{cor:W_I-max}.

\item Suppose $I \in \Q_w$. As in (b), this is equivalent to $\max W_I = w$. The description of $u_I = \max W_I$ in Remark~\ref{rem:u_I-calc} shows that $u_I$ is the union of two decreasing subsequences, and so it avoids $123$.




\item Let $\tilde{\Q}_w$ be the set on the right-hand side of \eqref{eq:Qw-matroid}, and suppose $I \in \tilde{\Q}_w$. Let us see that $w = u_I$, which implies $I \in \Q_w$. By part (c), $u_I$ avoids $123$. As mentioned in the proof of Lemma~\ref{lem:u_I-vs-v_I}, a $123$-avoiding permutation $z$ is completely determined by the set $\{(j, z(j)) : z(j) < n-j+1\}$. Hence, it is enough to show that if $w(j) < n-j+1$ or $u_I(j) < n-j+1$, then $w(j) = u_I(j)$. By Lemma~\ref{lem:u_I-vs-v_I}, this is equivalent to the claim that if $w(j) < n-j+1$ or $v_I(j) < n-j+1$, then $w(j) = v_I(j)$.

Observe that $v_I(j) < n{-}j{+}1$ if and only if $n+j \notin I$ and $\#(I \cap [v_I(j), n+j]) = j$. Thus, we want either of $w(j) < n{-}j{+}1$ or $v_I(j) < n{-}j{+}1$ to imply $n+j \notin I$ and $\#(I \cap [w(j), n{+}j]) = j$. The first condition is easy: if $n+j \in I$, then (1) $w(j) \geq n{-}j{+}1$ because $j$ must be a right-to-left maximum of $w$, and (2) the construction of $v_I$ implies $v_I(j) > n$.

Let $A$ be the set of anti-fixed points of $w$. Then $I$ is the disjoint union of $w(L(w))$, $R(w)+n$, and $I \cap Z(A)$, and we consider these three pieces of $I$ separately.
\begin{itemize}
    \item $\#(w(L(w)) \cap [w(j), n+j]) = \#(w(L(w)) \cap [w(j), n]) = \#(L(w) \cap [j])$, where the second equality uses the fact that $j \in L(w)$.
\item  $\#((R(w) + n) \cap [w(j), n+j]) = \#(R(w) \cap [j])$.
\item Lemma~\ref{lem:LRmin-description} implies that $[w(j), n+j] \supseteq [n-j+1, n+j]$,
\begin{align*}
    \#(I \cap Z(A) \cap [w(j), n+j]) &\geq \#(I \cap Z(A) \cap [n-j+1, n+j])\\
    &= \#(I \cap Z(A \cap [j])) \geq \#(A \cap [j]),
\end{align*}
where the last inequality follows from the description of $\tilde{\Q}_w$ from Remark~\ref{rem:Qw-description}.
\end{itemize}
Putting these three pieces of $I \cap [w(j), n+j]$ together,
\begin{align*}
\#(I \cap [w(j), n+j])) &= \#(L(w) \cap [j]) + \#(R(w) \cap [j]) + \#(I \cap Z(A) \cap [w(j), n+j])\\
&\geq \#(L(w) \cap [j]) + \#(R(w) \cap [j]) + \#(A \cap [j]) = j.
\end{align*}

For the reverse inequality, we use the easy direction of Hall's marriage theorem. Let $D_i$ be the set of *'s in column $i$ of the matrix $N_w$ representing $\tilde{\Q}_w$ (cf. Remark~\ref{rem:Qw-description}). That is,
\begin{equation*}
D_i = \begin{cases}
\{w^{-1}(i)\} & \text{if $i \in w(L(w))$}\\
\{i-n\} & \text{if $i \in R(w)+n$}\\
[a_{\ell}-\ell+k, a_{\ell}] & \text{if $i$ or $2n-i+1$ is $k$\th in a run of anti-fixed points $a_1, \ldots, a_{\ell}$}.
\end{cases}
\end{equation*}
Since $N_w$ has a transversal in columns $I$, we must have
\begin{align*}
\#(I \cap [w(j), n+j]) \leq \#\left(\bigcup_{i \in I \cap [w(j), n+j]} D_i\right).
\end{align*}
Notice that if $w(i) < i' \leq n$ or $n > i' > i+n$, the contents of $D_{i'}$ are bounded above by $i$. Therefore $\bigcup_{i \in I \cap [w(j), n+j]} D_i \subseteq [j]$, and we get the desired inequality.

We have now shown that $\tilde{\Q}_w \subseteq \Q_w$ for all $123$-avoiding $w$. By parts (a) and (b), the non-empty $\Q_w$ partition $\P_{12\cdots n} = {[2n] \choose n}$. Thus to get $\tilde{\Q}_w = \Q_w$, it is enough to show that
\begin{equation*}
\sum_{\substack{w \in S_n\\ \text{$w$ avoids $123$}}} \#\tilde{\Q}_w = {2n \choose n},
\end{equation*}
which we have done in Corollary~\ref{cor:123-avoiding-identity}.

\end{enumerate}

\end{proof}

Because $\#\P_{K, n} = C_{\#K+1}$, we get an immediate enumerative corollary.
\begin{cor} The size of $\P_w$ is
\begin{equation*}
\sum_{\substack{v \geq w\\ \text{$v$ avoids $123$}}} C_{\ell_1+1} \cdots C_{\ell_k+1},
\end{equation*}
where $\ell_1, \ldots, \ell_k$ are as in the statement of Theorem~\ref{thm:matroid-structure}, the lengths of the runs of anti-fixed points in each $v$.
\end{cor}

We conclude this section with a few results on symmetries of positroids which will be useful later. For $x \in [2n]$, write $\bar{x} \eqdef 2n+1-x$.

\begin{thm} \label{thm:reflection} $\overline{\P}_w = \P_{w_0 w^{-1} w_0}$ for any $w \in S_n$. \end{thm}

\begin{proof}
This follows from the matrix identity $w_0 w^{-1} M_w w_0^{(2n)} = M_{w_0 w^{-1} w_0}$, where $w_0$ is the reverse permutation in $S_n$ and $w_0^{(2n)}$ is the reverse permutation in $S_{2n}$.
\end{proof}

\begin{cor} \label{cor:reflection} $u_{\overline{I}} = w_0 u_I^{-1} w_0$ for any $I \in {[2n] \choose n}$. \end{cor}

\begin{proof}
Theorem~\ref{thm:reflection} is equivalent to $W_{\overline{I}} = w_0 W_I^{-1} w_0$. Since $w \mapsto w_0 w^{-1} w_0$ is an automorphism of Bruhat order,
\begin{equation*}
u_{\overline{I}} = \max W_{\overline{I}} = w_0(\max W_I)^{-1} w_0 = w_0 u_I^{-1} w_0.
\end{equation*}
\end{proof}

If $M$ is a matroid with groundset $E$, then $\{E \setminus I : \text{$I$ a basis of $M$}\}$ is also the set of bases for a matroid, the \emph{dual matroid} $M^*$.

\begin{thm} \label{thm:dual} $\P_w^*$ is isomorphic to $\P_{w^{-1}}$ for any permutation $w$. \end{thm}

\begin{proof}
Set $w_n^* = (n+1)\cdots(2n)1\cdots n \in S_n$. Let us see that $w_n^* \P_w^* = \P_{w^{-1}}$. Since inversion is an automorphism of Bruhat order and $\P_w = \bigcup_{v \geq w} \Q_v$ by Theorem~\ref{thm:matroid-structure}, it is enough to show that $w_n^* \Q_v^* = \Q_{v^{-1}}$. Let $A$ be the set of anti-fixed points of $v$, and $A_1, \ldots, A_k$ the maximal intervals in $A$. Since $[n]$ is the disjoint union $L(v) \cup R(v) \cup A$, we have
\begin{equation*}
\Q_v^* = \{vR(v)\} \oplus \{L(v) + n\} \oplus \bigoplus_{i=1}^k \P_{A_i, n}^*.
\end{equation*}
Also, $L(v) = v^{-1}L(v^{-1})$ and $R(v) = v^{-1}R(v^{-1})$, so
\begin{equation*}
w_n^* \Q_v^* = \{v^{-1}L(v^{-1})\} \oplus \{R(v^{-1}) + n\} \oplus \bigoplus_{i=1}^k w_n^* \P_{A_i, n}^*.
\end{equation*}
The set of anti-fixed points of $v^{-1}$ is $n+1-A$, so all we need to do is show $w_n^* \P_{K, n}^* = \P_{n+1-K,n}$ for any interval $K \subseteq [n]$. When pushed through the isomorphism of $\P_{K,n}$ with $\C_{\#K+1}$ given at the beginning of this section, this identity becomes $w_0 \C_{\#K+1}^* = \C_{\#K+1}$, where $w_0 \in S_{\#K+1}$. But the latter identity is certainly true: it reflects the existence of the automorphism of the set of Dyck paths which reverses the path and interchanges the notions of upstep and downstep.
\end{proof}

\section{The Tutte polynomial of $\P_w$}
\label{sec:tutte-polynomial}

Theorem~\ref{thm:matroid-structure} writes $\P_w$ as the disjoint union of matroids $\Q_v$ over 123-avoiding permutations $v$ above $w$ in Bruhat order, with each $\Q_v$ isomorphic to a direct sum of Catalan matroids and a matroid with one basis. In this section we give an analogous formula for the Tutte polynomial of $\P_w$, writing it as a sum over 123-avoiding permutations $v$ above $w$ of certain modifications of the Tutte polynomials of the $\Q_v$. First we recall one definition of the Tutte polynomial.

\begin{defn}
Given a matroid $M$ with groundset $S$, the \emph{rank} of a subset $I \subseteq S$ is the maximal size of an intersection of $I$ with a basis of $M$. Write $\rank_M(I)$ for this number. The \emph{Tutte polynomial} of $M$ is then the bivariate generating function
\begin{equation*}
T_M(x,y) = \sum_{I \subseteq S} (x-1)^{\rank(M) - \rank_M(I)} (y-1)^{\#I - \rank_M(I)}.
\end{equation*}
Here $\rank(M)$ is the size of any basis of $M$.
\end{defn}

Let $T_n(x,y)$ be the Tutte polynomial of the matroid $\P_n$. If $M$ is the matroid on $\{2n+1,2n+2\}$ with bases $\{\{2n+1\}\}$, then $M \oplus \P_n$ is isomorphic to $\C_{n+1}$. The Tutte polynomial of $M$ is $xy$, and Tutte polynomials are multiplicative on direct sums, so $T_n(x,y)$ is the Tutte polynomial of $\C_{n+1}$ divided by $xy$.

Given a Dyck path $D$, let $\height(D)$ be the height of the first peak and $\touch(D)$ the number of times $D$ touches the $x$-axis, not counting the first. In \cite{catalan-matroid}, Ardila shows that
\begin{equation*} 
\sum_{D \in \D_n} x^{\height(D)} y^{\touch(D)}.
\end{equation*}
is the Tutte polynomial of $\C_n$. Hence
\begin{equation} \label{eq:ardila-tutte}
T_n(x,y) = \sum_{D \in \D_{n+1}} x^{\height(D)-1} y^{\touch(D)-1}.
\end{equation}

It is more natural to give $T_n(x,y)$ as a sum over $\P_n$ using the bijection to $C_{n+1}$ given at the beginning of Section~\ref{sec:Pw-structure}. Define a total order $\prec$ on $[2n]$ by
\begin{equation*}
n+1 \prec n \prec n+2 \prec n-1 \prec \cdots \prec 2n \prec 1.
\end{equation*}
For $I \in \P_n$, define $c(I)$ as the length of the longest $\prec$-initial segment of $[2n]$, and $d(I)$ as the number of integers $j \in [2n]$ such that $\#(I \cap [n+1,n+j]) = \#(I \cap [n-j+1,n-1])$. Then Ardila's formula \eqref{eq:ardila-tutte} translates to
\begin{equation} \label{eq:catalan-tutte-polynomial}
T_n(x,y) = \sum_{I \in \P_n} x^{c(I)} y^{d(I)}.
\end{equation}

Given an interval $K \subseteq [n]$, define a modified version of $T_n$ as follows:
\begin{equation*}
T_{K,n}(x,y) = \begin{cases}
T_{\#K}(x,y) & \text{if $K = [n]$}\\
T_{\#K}(x,1) & \text{if $1 \in K$ and $n \notin K$}\\
T_{\#K}(1,y) & \text{if $1 \notin K$ and $n \in K$}\\
T_{\#K}(1,1) & \text{if $1, n \notin K$}
\end{cases}
\end{equation*}
Notice that $T_{\#K+1}(1,1) = C_{\#K+1}$, the number of bases in $\P_{K,n}$. Also, given a $123$-avoiding $w \in S_n$ with runs of anti-fixed points $A_1, \ldots, A_k$, define
\begin{equation*}
U_w(x,y) = \prod_{i=1}^k T_{A_i, n}(x,y).
\end{equation*}

\begin{thm} \label{thm:tutte-polynomial} For any permutation $w \in S_n$, the Tutte polynomial of $\P_w$ is
\begin{equation*}
U_{w_0}(x,y) + (1 - (x-1)(y-1))\sum_{\substack{w \leq v < w_0 \\ \text{$v$ avoids $123$}}} U_v(x,y).
\end{equation*}
\end{thm}

We start with a characterization of ranks in $\P_w$. Recall that Lemma~\ref{lem:positroid-simplification-2} associates to each $n$-subset $I$ of $[2n]$ an permutation $u_I \in S_n$ in such a way that $I \in \P_w$ if and only if $u_I \geq w$. We will follow a similar strategy here, and construct, for any nonnegative integer $r$ and any $I \subseteq [2n]$, a permutation $u_I^r$ such that $I$ has rank at least $r$ in $\P_w$ if and only if $u_I^r \geq w$.

Say $I \subseteq [2n]$ has size at least $r$. Define $J_r(I)$ to be the $\preceq$-lexicographically smallest $n$-set such that $\#(J_r(I) \cap I) \geq r$. Explicitly, if
\begin{equation*}
I = \{i_1 \prec i_2 \prec \cdots \} \qquad \text{ and} \qquad [2n] \setminus \{i_1, \ldots, i_r\} = \{j_1 \prec j_2 \prec \cdots \},
\end{equation*}
then $J_r(I) = \{i_1, \ldots, i_r, j_1, \ldots, j_{n-r}\}$. Now define $u_I^r = u_{J_r(I)}$.

\begin{thm} \label{thm:rank-characterization} The set $I$ has rank at least $r$ in $\P_w$ if and only if $J_r(I) \in \P_w$, or equivalently, $u_I^r \geq w$. \end{thm}

\begin{rem}
What is really important here is the partial order
\begin{equation*}
    n+1, n \prec n+2, n-1 \prec \cdots \prec 2n, 1.
\end{equation*}
One can show that although $J_r(I)$ depends on the choice of linear extension of this partial order to a total order, $u_I^r$ does not (indeed, this is a consequence of Theorem~\ref{thm:rank-characterization}).
\end{rem}

We postpone the proof of Theorem~\ref{thm:rank-characterization} since it is somewhat involved, and move on to its consequences for ranks in $\P_w$. Let $\P_w^r = \{I \subseteq [2n] : \text{$I$ has rank at least $r$ in $\P_w$}\}$, and $\Q_w^r = \P_w^r  \setminus \bigcup_{v > w} \P_v^r$. Theorem~\ref{thm:rank-characterization} shows that $\Q_w^r$ is the set of $I$ such that $u_I^r = w$. Equivalently, if we think of $J_r$ as a function $2^{[2n]} \to {[2n] \choose n}$, then $\Q_w^r = \bigcup_{K \in \Q_w} J_r^{-1}(K)$, and we can give a reasonable description of $J_r^{-1}(K)$ for a fixed $K$.

\begin{lem} \label{lem:Jk-fiber} Let $K$ be an $n$-subset of $[2n]$, and $0 \leq r \leq n$. Write $K = E \cup F$ where $E$ is the maximal initial segment of $[2n]$ in $K$ (in the order $\prec$). Then $J_r^{-1}(K)$ is the collection of sets of the form $E' \cup F \cup G$, where $E' \in {E \choose \#E-n+r}$ and $G \subseteq [2n]$ satisfies $\min(G) > \max(F)$.
\end{lem}

\begin{proof} Say $I = E' \cup F \cup G$ where $E'$, $F$, $G$ are as in the statement of the lemma, and write $I = \{i_1 \prec i_2 \prec \cdots\}$. Since $\#E' + \#F = r$, we have $\{i_1, \ldots, i_r\} = E' \cup F$. Thus $[2n] \setminus \{i_1, \ldots, i_r\}$ contains $E \setminus E'$, which has size $n-r$. Since $E$ is an initial segment, the smallest $n-r$ elements of $[2n] \setminus \{i_1, \ldots, i_r\}$ are exactly $E \setminus E'$, so $J_r(I) = (E' \cup F) \cup (E \setminus E') = K$.

Conversely, suppose $J_r(I) = K$, with $[2n] \setminus I = \{j_1 \prec j_2 \prec \cdots\}$ as in the definition of $J_r$. Let $E'$ consist of the $\prec$-first $\#E-n+r$ elements of $I$ (noting that $\#E-n+r \leq r \leq \#I$). Since $E$ is an initial segment of size $\#E'+n-r$, we must have $E' \cup \{j_1 \prec \cdots \prec j_{n-r}\} = E$. But this forces $F = \{i_{\#E'+1} \prec \cdots \prec i_r\} \subseteq I$, and then defining $C = \{i_{r+1} \prec \cdots \prec i_n\}$ gives the desired decomposition $I = E' \cup F \cup G$.
\end{proof}

Finally, we will need a description of $U_w$ in the style of $\eqref{eq:catalan-tutte-polynomial}$. As above, let $c(K)$ be the length of the largest $\prec$-initial segment of $[2n]$ contained in $K$, and let $\bar{c}(K)$ be the length of the largest $\prec$-final segment in $[2n] \setminus K$.

\begin{lem} \label{lem:modified-tutte} For any $123$-avoiding $w \neq w_0$,
\begin{equation*}
U_w(x,y) = \sum_{K \in \Q_w} x^{c(K)} y^{\overline{c}(K)}.
\end{equation*}
\end{lem}

\begin{proof}
Suppose $w$ avoids $123$, and has runs of anti-fixed points $A_1, \ldots, A_k$. Then any $K \in \Q_w$ is a disjoint union
\begin{equation*}
L_1 \cup \cdots \cup L_k \cup w(L(w)) \cup (R(w)+n),
\end{equation*}
where $L_i \in \P_{A_i,n}$.

Suppose $K \in \Q_w$ contains as a maximal $\prec$-initial segment $E = \{n+1, \ldots, n+\alpha, n, n-1, \ldots, n-\beta+1\}$ for some $\alpha, \beta$. By definition of $\Q_w$, this means $w$ has right-to-left maxima in positions $1, \ldots, \alpha$. But this is only possible if $w$ has anti-fixed points in those positions. Likewise, $w$ has left-to-right minima with values $n, n-1, \ldots, n-\beta+1$, hence anti-fixed points in positions $1, \ldots, \beta$. This shows that $E \subseteq L_1$ if $w(1) = n$, and that $E = \emptyset$ if $w(1) \neq n$. Hence $c(K) = c(L_1)$ if $w(1) = n$, and $c(K) = 0$ otherwise. An analogous argument shows that $\bar{c}(K) = \bar{c}(L_r)$ if $w(n) = 1$, and $\bar{c}(K) = 0$ otherwise. Now we see that:
\begin{itemize}
\item If $w(1) \neq n$ and $w(n) \neq 1$, then
\begin{equation*}
\sum_{K \in \Q_w} x^{c(K)} y^{\overline{c}(K)} = \#\Q_w = \prod_{i=1}^k T_{\#A_i}(1,1) = U_w(x,y).
\end{equation*}

\item If $w(1) = n$ and $w(n) \neq 1$, then using \eqref{eq:catalan-tutte-polynomial},
\begin{align*}
\sum_{K \in \Q_w} x^{c(K)} y^{\overline{c}(K)} &= \sum_{L \in \P_{\#A_1}} x^{c(L)} \prod_{i=2}^k T_{\#A_i}(1,1) \\
&=  T_{\#A_1}(x,1)\prod_{i=2}^k T_{\#A_i}(1,1) = U_w(x,y).
\end{align*}

\item If $w(1) \neq n$ and $w(n) = 1$, then
\begin{align*}
\sum_{K \in \Q_w} x^{c(K)} y^{\overline{c}(K)} &= \sum_{L \in \P_{\#A_k}} y^{\overline{c}(L)} \prod_{i=1}^{k-1} T_{\#A_i}(1,1)\\
&= \sum_{L \in \P_{\#A_k}} y^{c(L)} \prod_{i=1}^{k-1} T_{\#A_i}(1,1)\\
&= T_{\#A_k}(y,1) \prod_{i=1}^{k-1} T_{\#A_i}(1,1)
\end{align*}
To get the second equality, we use the fact from Theorem~\ref{thm:dual} that $I \mapsto w_n^*([2n] \setminus I)$ is an automorphism of $\P_n$, and that it exchanges the statistics $c$ and $\bar{c}$. Taking the dual of a matroid corresponds to switching the variables in the Tutte polynomial, so $T_n(x,y) = T_n(y,x)$ since $\P_n$ is self-dual. Thus
\begin{equation*}
\sum_{K \in \Q_w} x^{c(K)} y^{\overline{c}(K)} = T_{\#A_k}(1,y) \prod_{i=1}^{k-1} T_{\#A_i}(1,1) = U_w(x,y).
\end{equation*}

\item If $w(1) = n$ and $w(n) = 1$, then $k > 1$ since $w \neq w_0$, and 
\begin{align*}
\sum_{K \in \Q_w} x^{c(K)} y^{\overline{c}(K)} &= \sum_{L \in \P_{\#A_1}} x^{c(L)} \sum_{L \in \P_{\#A_k}} y^{\overline{c}(L)} \prod_{i=2}^{k-1} T_{\#A_i}(1,1) \\
&= T_{\#A_1}(x,1) T_{\#A_k}(1,y)  \prod_{i=2}^{k-1} T_{\#A_i}(1,1) = U_w(x,y).
\end{align*}
\end{itemize}
\end{proof}

Let $T_w(x,y)$ be the Tutte polynomial of $\P_w$. Recall that Theorem~\ref{thm:tutte-polynomial} claims that
\begin{equation*}
    T_w(x,y) = U_{w_0}(x,y) + (1 - (x-1)(y-1))\sum_{\substack{w \leq v < w_0 \\ \text{$v$ avoids $123$}}} U_v(x,y).
\end{equation*}
    The M\"obius function of Bruhat order on $S_n$ is $\mu(w,v) = (-1)^{\ell(v) - \ell(w)}$. By M\"obius inversion, for any particular $w \in S_n$, Theorem~\ref{thm:tutte-polynomial} is equivalent to
\begin{equation*}
\sum_{v \geq w} (-1)^{\ell(v) - \ell(w)} T_v(x,y) = \begin{cases}
U_{w_0}(x,y) = T_n(x,y) & \text{if $w = w_0$}\\
(1 - (x-1)(y-1))U_w & \text{if $w \neq w_0$ avoids $123$}\\
0 & \text{if $w$ contains $123$}
\end{cases}.
\end{equation*}

\begin{proof}[Proof of Theorem~\ref{thm:tutte-polynomial}]
Write $\rank_v(I)$ for the rank of $I$ in $\P_v$. By Theorem~\ref{thm:rank-characterization}, $\rank_v(I) = r$ if and only if $v \leq u_I^r$ and $v \not\leq u_I^{r+1}$. Thus,
\begin{align*}
\sum_{v \geq w} (-1)^{\ell(v) - \ell(w)} T_v(x,y) &= \sum_{v \geq w} (-1)^{\ell(v) - \ell(w)} \sum_{I \subseteq [2n]} (x-1)^{n-\rank_v(I)} (y-1)^{\#I-\rank_v(I)}\\
&= \sum_{I \subseteq [2n]} \sum_{r=0}^n (x-1)^{n-r} (y-1)^{\#I-r} \sum_{v \in [w, u_I^r] \setminus [w, u_I^{r+1}]} (-1)^{\ell(v) - \ell(w)}.
\end{align*}
The term $(x-1)^{n-r}(y-1)^{\#I-r}$ will occur frequently, so we will simply write $f$ for it in the rest of the proof.

Any Bruhat interval with more than one element has the same number of elements of even length and of odd length \cite{bjorner-brenti}, so
\begin{equation*}
\sum_{v \in [w, u_I^r] \setminus [w, u_I^{r+1}]} (-1)^{\ell(v) - \ell(w)} = \begin{cases}
0 & \text{if $w \neq u_I^r$ and $w \neq u_I^{r+1}$}\\
1 & \text{if $w = u_I^r > u_I^{r+1}$}\\
-1 & \text{if $u_I^r > u_I^{r+1} = w$}
\end{cases}
\end{equation*}
Observe that $w = u_I^r > u_I^{r+1}$ if and only if $I \in \Q_w^r \setminus \Q_w^{r+1}$, and $u_I^r > u_I^{r+1} = w$ if and only if $I \in \Q_w^{r+1} \setminus \Q_w^r$. Therefore
\begin{align*}
\sum_{v \geq w} (-1)^{\ell(v) - \ell(w)} T_v(x,y) &= \sum_{r=0}^n \left[\sum_{I \in \Q_w^r \setminus \Q_w^{r+1}} f - \sum_{I \in \Q_w^{r+1} \setminus \Q_w^r} f \right]\\
&= \sum_{r=0}^n \left[\sum_{I \in \Q_w^r} f - \sum_{I \in \Q_w^{r+1}} f\right]\\
&= \sum_{r=0}^n \sum_{I \in \Q_w^r} f - (x-1)(y-1)\sum_{r=1}^{n+1} \sum_{I \in \Q_w^r} f.
\end{align*}

We may as well assume $w \neq w_0$, in which case $\Q_w^0 = \emptyset$. Also, $\Q_w^{n+1} = \emptyset$ for any $w$, so
\begin{align*}
\sum_{v \geq w} (-1)^{\ell(v) - \ell(w)} T_v(x,y) &= [1 - (x-1)(y-1)]\sum_{r=0}^n \sum_{I \in \Q_w^r} f\\
&= [1 - (x-1)(y-1)]\sum_{r=0}^n \sum_{K \in \Q_w} \sum_{I \in J_r^{-1}(K)} f.
\end{align*}
As in Lemma~\ref{lem:modified-tutte}, let $c(K)$ denote the length of the largest initial segment of $[2n]$ in $K$ in the order $\prec$, and $\bar{c}(K)$ the length of the largest final segment in $[2n] \setminus K$. By Lemma~\ref{lem:Jk-fiber}, a member of $J_r^{-1}(I)$ with size $j+r$ corresponds to a choice of (1) a $(c(K)-n+r)$-subset of a set of size $c(K)$, and (2) a $j$-subset of the maximal $\prec$-final segment of $[2n] \setminus K$. Hence
\begin{align*}
\sum_{I \in J_r^{-1}(K)} f &= \sum_{I \in J_r^{-1}(K)} (x-1)^{n-r}(y-1)^{\#I-r}\\
 &= {c(K) \choose c(K)-n+r} (x-1)^{n-r} \sum_{j=0}^{\overline{c}(K)} {\bar{c}(K) \choose j} (y-1)^j\\
 &= {c(K) \choose c(K)-n+r} (x-1)^{n-r} y^{\bar{c}(K)}.
\end{align*}

Continuing on,
\begin{align*}
\sum_{v \geq w} (-1)^{\ell(v) - \ell(w)} T_v(x,y) &= [1 - (x-1)(y-1)] \sum_{K \in \Q_w} \sum_{r=0}^n {c(K) \choose c(K)-n+r} (x-1)^{n-r} y^{\bar{c}(K)}\\
&= [1 - (x-1)(y-1)] \sum_{K \in \Q_w} \sum_{r=0}^n {c(K) \choose c(K)-n+r} (x-1)^{n-r} y^{\bar{c}(K)}\\
\end{align*}
This is equal to $[1 - (x-1)(y-1)]U_w(x,y)$ by Lemma~\ref{lem:modified-tutte}.
\end{proof}

To prove Theorem~\ref{thm:rank-characterization}, we will need some lemmas giving a Bruhat relation between $u_I$ and $u_J$ for two sets $I$ and $J$. Write $I \unlhd J$ if $I$ is the $\prec$-lexicographically minimal $\#I$-subset of $I \cup J$. Equivalently, $I \unlhd J$ if and only if $J \cap I = J \cap [\max_{\prec}(I)]$. This is a partial order on finite subsets of $\N$ of a fixed size.

\begin{lem} \label{lem:respects-bruhat} Suppose $I, J \in {[2n] \choose n}$ are such that either
\begin{enumerate}[(a)]
\item $I \unlhd J$, or
\item $J = I \setminus \{i\} \cup \{j\}$, where $i$ is contained in a $\prec$-initial segment in $I$ and $i \preceq j$.
\end{enumerate}
Then $u_I \geq u_J$.
\end{lem}

\begin{proof}
For the case where $I \unlhd J$, we may assume that $J = I \setminus \{i\} \cup \{j\}$ where $i \in I$ and $j \succ \max_{\prec}(I)$, since this is the covering relation for $\unlhd$. Recall the injective word $v_I$, with the property that $u_I$ is the greatest lower bound of $v_I$ and $w_0$, and whose entries are $I \cap [n+1,2n]$ in increasing order together with $I \cap [n]$ in decreasing order.

Suppose for the moment that $j \leq n$. In passing from $v_I$ to $v_J$, we remove one entry ($i$), insert a new entry ($j$) into the decreasing subsequence formed by $I \cap [n]$ in the unique way that keeps the subsequence decreasing, and then shift part of the subsequence either right or left to fill the gap left by $i$. If $i \leq n$, then $j \succ i$ implies $j < i$. Thus, $j$ enters right of the gap left by $i$, so we shift leftward. This means that $v_J$ is entrywise less than or equal to $v_I$, which implies the weaker statement that $v_J \leq v_I$ in Bruhat order. Therefore $u_J \leq u_I$.

Next suppose that $j \leq n$ still, but now $i > n$. We consider cases (a) and (b) separately. In case (b), where $i$ is contained in a $\prec$-initial segment in $I$, $v_I$ begins $(n+1)(n+2)\cdots (n+b)\cdots$, with $i$ being one of those first $b$ entries. Thus, every entry of the decreasing sequence is right of $i$, and in particular $j$ does enter to the right of it when we pass to $v_J$. In case (a), we have $j \succeq \max_{\prec}(I)$, which implies $j \leq \min(I)$ (in the usual order), so $j$ will be the last entry in the decreasing sequence in $v_J$. In particular, $j$ enters right of the gap where $i$ was. In both cases we end up with $v_J$ entrywise less than or equal to $v_I$ as before, as in the last paragraph.

Finally, assume that $j > n$. We will apply the map $x \mapsto \overline{x} = 2n+1-x$ and use Corollary~\ref{cor:reflection}. The arguments above only depend on $\prec$ being a linear extension of the partial order
\begin{equation*}
n+1,n \prec n+2,n-1 \prec \cdots \prec 2n,1
\end{equation*}
and so they still go through if we replace $\prec$ with the total order $\mathrel{\overline{\prec}}$ defined by
\begin{equation*}
n \mathrel{\overline{\prec}} n+1 \mathrel{\overline{\prec}} n-1 \mathrel{\overline{\prec}} n+2 \mathrel{\overline{\prec}} \cdots \mathrel{\overline{\prec}} 1 \mathrel{\overline{\prec}} 2n.
\end{equation*}
The hypotheses of the lemma still hold for $\overline{I}$, $\overline{i}$, and $\overline{j}$ using the order $\mathrel{\overline{\prec}}$.

As $\overline{j} \leq n$, the previous arguments show that $u_{\overline{I}} \geq u_{\overline{J}}$, or $w_0 u_I^{-1} w_0 \geq w_0 u_J^{-1} w_0$ by Corollary~\ref{cor:reflection}. Since $w \mapsto w_0 w^{-1} w_0$ is an automorphism of Bruhat order, this is equivalent to $u_I \geq u_J$.
\end{proof}

\begin{lem} \label{lem:Jk-respects-diamond} Say $I, I' \in {[2n] \choose r}$, where $r \leq n$. If $I \unlhd I'$, then $J_r(I) \unlhd J_r(I')$. \end{lem}

\begin{proof}
As in the proof of Lemma~\ref{lem:respects-bruhat}, we can assume that $I' = I \setminus \{i\} \cap \{j\}$, where $i \in I$ and $j \succ \max_{\prec}(I)$. Write $I = \{i_1 \prec \cdots \prec i_r\}$ and $[2n] \setminus I = \{j_1 \prec j_2 \prec \cdots\}$, so $J_r(I) = \{i_1, \ldots, i_r, j_1, \ldots, j_{n-r}\}$. There are several cases.

\begin{itemize}
\item If $j \preceq j_{n-r}$, then $J_r(I') = J_r(I)$.
\item If $i \preceq j_{n-r} \prec j$, then $J_r(I') = J_r(I) \setminus \{j_{n-r}\} \cup \{j\}$. Here $j \succ \max_{\prec}(I)$ and $j \succeq j_{n-r}$, so $j \succ \max_{\prec} J_r(I)$.
\item If $j \prec i \preceq j$, then $J_r(I') = J_r(I) \setminus \{i\} \cup \{j\}$. Once again, $j \succ \max_{\prec}(I)$ and $j \succ i \succeq j_{n-r}$, so $j \succ \max_{\prec} J_r(I)$.
\end{itemize}
\end{proof}

We can now prove Theorem~\ref{thm:rank-characterization}; recall it claims that $I \in \P_w$ has rank $\geq r$ if and only if $u_I^r \geq w$.

\begin{proof}[Proof of Theorem~\ref{thm:rank-characterization}]
Define $W_I^r = \{w \in S_n : I \in \P_w^r\}$.  It is clear from Theorem~\ref{thm:matroid-structure} that $W_I^r$ is a lower order ideal in Bruhat order. Theorem~\ref{thm:rank-characterization} is equivalent to the assertion that $u_I^r$ is the unique maximal element of $W_I^r$.

First we reduce to the case where $\#I = r$. Notice that $I \in \P_w^r$ if and only if $I' \in \P_w^r$ for some $r$-subset $I'$ of $I$. Equivalently,
\begin{equation*}
W_I^r = \bigcup_{I' \in {I \choose r}} W_{I'}^r.
\end{equation*}
If $I'$ is the $\prec$-lexicographically least $r$-subset of $I$, then $J_r(I') = J_r(I)$, so $u_{I'}^r = u_I^r$. For any other $r$-subset $I''$ of $I$, we have $I' \unlhd I''$. Lemma~\ref{lem:Jk-respects-diamond} then says $J_r(I') \unlhd J_r(I'')$, so Lemma~\ref{lem:respects-bruhat} implies $u_{I'}^r \geq u_{I''}^r$. Thus if we knew that each $W_{I''}^r$ has $u_{I''}^r$ as a unique maximum, we would be done: the unique maximum of $W_I^r$ would be $u_{I'}^r$. In other words, we can assume $\#I = r$.

Now we induct (downward) on $r$, assuming $\#I = r$. If $r = n$, then $W_I^r = W_I$ has $u_I^r = u_I$ as its unique maximum by Theorem~\ref{thm:matroid-structure}. Suppose $r < n$. Then $I \in \P_w^r$ if and only if $I \cup x \in \P_w^{r+1}$ for some $x \notin I$, or equivalently,
\begin{equation*}
W_I^r = \bigcup_{x \notin I} W_{I \cup x}^{r+1}.
\end{equation*}
By induction, each $W_{I \cup x}^{r+1}$ has $u_{I \cup x}^{r+1}$ as its unique maximal element. What we want to show, therefore, is that if $x \notin I$, then $u_I^r \geq u_{I \cup x}^{r+1}$, with equality holding for some $x$.

As in the definition of $J_r(I)$, write $I = \{i_1 \prec \cdots \prec i_r\}$ and $[2n] \setminus I = \{j_1 \prec j_2 \prec \cdots\}$, so that $J_r(I) = \{i_1, \ldots, i_r, j_1, \ldots, j_{n-r}\}$. Then
\begin{equation*}
J_{r+1}(I \cup x) = \begin{cases}
J_r(I) & \text{if $x \preceq j_{n-r}$}\\
J_r(I) \setminus \{j_{n-r}\} \cup \{x\} & \text{if $x \succ j_{n-r}$}
\end{cases}
\end{equation*}
In particular, if $x$ is $\prec$-minimal in $[2n] \setminus I$, then $x \preceq j_{n-r}$, so $u_I^r = u_{I \cup x}^{r+1}$.

We can now assume that $x \succ j_{n-r}$. By definition, $j_{n-r}$ is part of a $\preceq$-initial segment in $J_r(I)$, so Lemma~\ref{lem:respects-bruhat} shows that
\begin{equation*}
u_I^r = u_{J_r(I)} \geq u_{J_{r+1}(I \cup x)} = u_{I \cup x}^{r+1}.
\end{equation*}
\end{proof}

\section{Transversal matroids associated to permutation diagrams} \label{sec:diagram-matroids}

In this section we give some conjectures to the effect that results like Theorem~\ref{thm:matroid-structure} and Theorem~\ref{thm:tutte-polynomial} hold for another family of rank $n$ matroids on $[2n]$ indexed by $S_n$.

\begin{defn} The \emph{Rothe diagram} of $w \in S_n$ is
\begin{equation*}
    D(w) \eqdef \{(i,w(j)) \in [n] \times [n] : i < j, w(i) > w(j)\}.
\end{equation*}
\end{defn}

Given $w \in S_n$, let $\tilde{M}_w$ be a generic $n \times 2n$ matrix $[I_n \mid A]$, where $I_n$ is an $n\times n$ identity matrix, and $A$ is $n \times n$ with $A_{ij} = 0$ whenever $(i,j) \in D(w)$. The \emph{diagram matroid} $DM_w$ of $w$ is the matroid of $\tilde{M}_w$.

\begin{ex}
Say $w = 31524$. Then
\begin{equation*}
D(w) = \begin{array}{ccccc}
\circ & \circ & \cdot & \cdot & \cdot\\
\cdot & \cdot & \cdot & \cdot & \cdot\\
\cdot & \circ & \cdot & \circ & \cdot\\
\cdot & \cdot & \cdot & \cdot & \cdot\\
\cdot & \cdot & \cdot & \cdot & \cdot
\end{array},
\end{equation*}
where we use matrix coordinates, and $\circ$ for lattice points in $D(w)$, $\cdot$ for those not in $D(w)$. The diagram matroid of $w$ is then the matroid of a generic matrix
\begin{equation*}
\begin{bmatrix}
1 & 0 & 0 & 0 & 0 & 0 & 0 & * & * & *\\
0 & 1 & 0 & 0 & 0 & * & * & * & * & *\\
0 & 0 & 1 & 0 & 0 & * & 0 & * & 0 & *\\
0 & 0 & 0 & 1 & 0 & * & * & * & * & *\\
0 & 0 & 0 & 0 & 1 & * & * & * & * & *
\end{bmatrix}.
\end{equation*}
\end{ex}

\begin{conj} \label{conj:diagram-matroid-basis-count}
Theorem~\ref{thm:Pw-formula} holds for $DM_w$. That is, for any $w \in S_n$, the number of bases of $DM_w$ is
\begin{equation*}
\sum_{\substack{v \geq w \\ \text{$v$ avoids $123$}}} C_{\ell_1+1} \cdots C_{\ell_k+1},
\end{equation*}
where $\ell_1, \ldots, \ell_k$ are the lengths of the runs of anti-fixed points of $v$.
\end{conj}

Theorem~\ref{thm:matroid-structure} no longer holds: it can happen that $w \leq v$ but $DM_v \not\subseteq DM_w$. One can still hope to prove Conjecture~\ref{conj:diagram-matroid-basis-count} by M\"obius inversion, but a less trivial sign-reversing involution would be required. Conjecture~\ref{conj:diagram-matroid-basis-count} would follow from a stronger conjecture on Tutte polynomials.
\begin{conj} \label{conj:diagram-matroid-tutte} For any $w \in S_n$, the Tutte polynomial of $DM_w$ is equal to the Tutte polynomial of $\P_w$. \end{conj}

 If $DM_w$ and $\P_w$ were isomorphic, then Conjecture~\ref{conj:diagram-matroid-tutte} would of course be true, but this need not be the case.

\begin{conj} The matroids $DM_w$ and $\P_w$ are isomorphic if and only if $w$ avoids the pattern $21354$. \end{conj}

These conjectures have all been verified through $S_7$. Despite this, their Tutte polynomials seem to agree, also verified through $S_7$.

There is a combinatorial procedure called \emph{shifting} that relates $DM_w$ and $\P_w$ (and which has geometric connections making it useful in studying positroid varieties and other subvarieties of Grassmannians \cite{knutson-interval-positroid-varieties, pawlowski-liu-conjecture-rank-varieties}). Given integers $i$ and $j$, and a set $I$, let
\begin{equation*}
\shift_{i\to j} I = \begin{cases}
I \setminus \{i\} \cup \{j\} & \text{if $i \in I$ and $j \notin I$}\\
I & \text{else}
\end{cases}
\end{equation*}
If $X$ is a collection of sets, and $I \in X$, then we define
\begin{equation*}
\shift_{i\to j, X} I = \begin{cases}
\shift_{i \to j} I & \text{if $\shift_{i \to j} I \neq I$ and $\shift_{i \to j} I \notin X$}\\
I & \text{else}
\end{cases}
\end{equation*}
Finally, define $\shift_{i\to j} X$ to be $\{\shift_{i\to j, X} I : I \in X\}$.

Let $\mathcal{B}(A)$ denote the set of bases of the matroid of a matrix $A$. We can also apply shifting to matrices. Let $\shift_{i\to j}A$ be the matrix of the same size as $A$ such that
\begin{equation*}
A_{pq} = \begin{cases}
A_{pi} & \text{if $q = j$ and $A_{pj} = 0$}\\
0      & \text{if $q = i$ and $A_{pj} = 0$}\\
A_{pq} & \text{else}
\end{cases}
\end{equation*}

We have $\#(\shift_{i \to j} X) = \#X$, but it need not be the case that $\mathcal{B}(\shift_{i \to j} A) = \shift_{i \to j} \mathcal{B}(A)$. For example, if $A$ is a $2\times 2$ identity matrix, then $\shift_{2 \to 1}\mathcal{B}(A) = \{12\}$, while $\mathcal{B}(\shift_{i \to j}A)$ is empty. In general, we only get a containment.

\begin{lem} \label{lem:matroid-shift-containment} If the entries of $A$ are algebraically independent, then $\mathcal{B}(\shift_{i\to j} A) \subseteq \shift_{i \to j}\mathcal{B}(A)$. \end{lem}

\begin{proof} Suppose $I \in \mathcal{B}(\shift_{i \to j} A)$, where $I = \{b_1 < \cdots < b_n\}$. Then there is a transversal of $A$ in columns $I$, i.e. a bijection $\pi : I \to [n]$ such that $(\shift_{i \to j} A)_{\pi(b_p)b_p} \neq 0$ for each $p$. We consider various cases.

\begin{itemize}
\item If $i, j \notin I$, then $I \in \mathcal{B}(A)$ and $\shift_{i \to j}I = I$, so $I \in \shift_{i \to j}\mathcal{B}(A)$.

\item If $i \in I$, $j \notin I$, then again $I \in \mathcal{B}(A)$, because $\shift_{i \to j}A$ restricted to columns $I$ is $A$ restricted to columns $I$ with some nonzero entries made zero. Since $(\shift_{i\to j}A)_{\pi(i)i}$ is nonzero, $A_{\pi(i)j}$ must be nonzero. Therefore the bijection $\pi' : I \setminus \{i\} \cup \{j\} \to [n]$ which agrees with $\pi$ on $I \setminus \{i\}$ and having $\pi'(j) = \pi(i)$ is a transversal of $A$. This shows that $\mathcal{B}(A)$ also contains $\shift_{i \to j}I$. But then $I \in \shift_{i \to j}\mathcal{B}(A)$.

\item Suppose $i \notin I$, $j \in I$. If  $A_{\pi(j)i} \neq 0$, then modifying $\pi$ appropriately as in the last case will give a transversal of $A$ in columns $I \setminus \{j\} \cup \{i\}$. Then $I = \shift_{i\to j}(I \setminus \{j\} \cup \{i\}) \in \shift_{i\to j}\B(A)$.

If $A_{\pi(j)i} = 0$, then $A_{\pi(j)j} = (\shift_{i \to j}A)_{\pi(j)j} \neq 0$, and so $I \in \B(A)$. Then $I = \shift_{i \to j}I \in \shift_{i\to j}\B(A)$.

\item Suppose $i, j \in I$. Since $(\shift_{i \to j}A)_{\pi(i)i}$ is nonzero, so is $A_{\pi(i)i}$. Therefore if $A_{\pi(j)j} \neq 0$, then $\pi$ is still a transversal of $A$ in columns $I$.

Now suppose $A_{\pi(j)j} = 0$. Then, since $(\shift_{i\to j}A)_{\pi(j)j}$ is nonzero, so is $A_{\pi(j)i}$. Also, since $(\shift_{i \to j}A)_{\pi(i)i}$ is nonzero, so is $A_{\pi(i)j}$. Therefore the bijection $\pi' : I \to [n]$ agreeing with $\pi$ on $I \setminus \{i,j\}$, and having $\pi'(i) = \pi(j)$, $\pi'(j) = \pi(i)$, is a transversal of $A$ in columns $I$.

Either way we see that $I \in \B(A)$, and so $I = \shift_{i\to j}I \in \shift_{i \to j}\B(A)$.
\end{itemize}
\end{proof}

The matrices $M_w$ and $\tilde{M}_w$ defining $\P_w$ and $DM_w$ turn out to be related by a sequence of shifts. Let $\shift_w$ be the composition $\shift_{2n \to w(n)} \cdots \shift_{n+2 \to w(2)} \shift_{n+1 \to w(1)}$.

\begin{lem}[\cite{pawlowski-liu-conjecture-rank-varieties}, Theorem 5.5] \label{lem:matrix-shift} For any permutation $w$, $\shift_w M_w = \tilde{M}_w$. \end{lem}

Thus, Lemma~\ref{lem:matroid-shift-containment} shows that $DM_w \subseteq \shift_w \P_w$. Since shifting preserves the size of a collection of sets, we see that Conjecture~\ref{conj:diagram-matroid-basis-count} is equivalent to:

\begin{conj} \label{conj:shifting} $\shift_w \P_w = DM_w$ for any permutation $w$. \end{conj}

\bibliographystyle{plain}
\bibliography{positroid_enumeration_new}

\begin{thebibliography}{10}

\bibitem{catalan-matroid}
Federico Ardila.
\newblock The {Catalan} matroid.
\newblock {\em Journal of Combinatorial Theory Series A}, 104(1):49--62, 2003.

\bibitem{rank-varieties}
Sara Billey and Izzet Coskun.
\newblock Singularities of generalized {Richardson} varieties.
\newblock {\em Comm. Algebra}, 40(4):1466--1495, 2012.

\bibitem{bjorner-brenti}
Anders {Bj\"orner} and Francesco Brenti.
\newblock {\em Combinatorics of Coxeter Groups}.
\newblock Springer, 2005.

\bibitem{bonin-transversal-matroids}
Joseph~E. Bonin.
\newblock An introduction to transversal matroids.
\newblock Retrieved in January 2015 from http://home.gwu.edu/\textasciitilde
  jbonin/TransversalNotes.pdf.

\bibitem{knutson-interval-positroid-varieties}
Allen Knutson.
\newblock Schubert calculus and shifting of interval positroid varieties.
\newblock 2014.
\newblock arXiv:1408.1261.

\bibitem{positroidjuggling}
Allen Knutson, Thomas Lam, and David Speyer.
\newblock Positroid varieties: Juggling and geometry.
\newblock {\em Compos. Math.}, 149:1710--1752, 2013.

\bibitem{krattenthaler-123-avoiding-dyck-path-bijection}
Christian Krattenthaler.
\newblock Permutations with restricted patterns and {Dyck} paths.
\newblock {\em Advances in Applied Mathematics}, 27:510--530, 2001.

\bibitem{oh-positroid-characterization}
Suho Oh.
\newblock Positroids and {Schubert} matroids.
\newblock {\em Journal of Combinatorial Theory Series A}, 118:2426--2435, 2011.

\bibitem{pawlowski-liu-conjecture-rank-varieties}
Brendan Pawlowski.
\newblock Cohomology classes of rank varieties and a conjecture of {Liu}.
\newblock 2014.
\newblock arXiv:1410.7419.

\bibitem{postnikov-positroids}
Alexander Postnikov.
\newblock Total positivity, {Grassmannians}, and networks.
\newblock arXiv:math/0609764, 2006.

\end{thebibliography}

\end{document}